\documentclass[review,onefignum,onetabnum]{siamart190516}

\usepackage{lipsum}
\usepackage{amsfonts}
\usepackage{graphicx}
\usepackage{epstopdf}
\usepackage{algorithmic}
\ifpdf
  \DeclareGraphicsExtensions{.eps,.pdf,.png,.jpg}
\else
  \DeclareGraphicsExtensions{.eps}
\fi

\usepackage{enumitem}
\setlist[enumerate]{leftmargin=.5in}
\setlist[itemize]{leftmargin=.5in}

\newsiamremark{remark}{Remark}
\newsiamremark{hypothesis}{Hypothesis}
\crefname{hypothesis}{Hypothesis}{Hypotheses}
\newsiamthm{claim}{Claim}

\headers{Fast algorithm for GNH matrix}{C. Chen, S. Reiz, C. Yu, H.-J. Bungartz, and G. Biros}

\title{Fast Approximation of the Gauss-Newton Hessian Matrix for the Multilayer Perceptron
}

\author{Chao Chen\thanks{University of Texas at Austin, United States (\email{chenchao.nk@gmail.com}, \email{biros@oden.utexas.edu}).}
\and Severin Reiz\thanks{Technical University of Munich, Germany (\email{s.reiz@tum.de}, \email{bungartz@tum.de}).}
\and Chenhan Yu\thanks{Nvidia Corp., United States (\email{b94201001@gmail.com}).}
\and Hans-Joachim Bungartz\footnotemark[3]
\and George Biros\footnotemark[2]}

\usepackage{amsopn}

\DeclareMathOperator{\bigO}{\mathcal{O}}
\DeclareMathOperator{\bigOT}{{\mathcal{O}}}

\ifpdf
\hypersetup{
  pdftitle={Approximation of GNH for MLP},
  pdfauthor={C. Chen, S. Reiz, C. Yu, H.-J. Bungartz, and G. Biros}
}
\fi

\usepackage{array}
\usepackage{booktabs}
\usepackage{comment}
\usepackage{siunitx}
\usepackage{nicefrac}
\usepackage{mathtools}
\usepackage[flushleft]{threeparttable}
\usepackage{subcaption}

\DeclarePairedDelimiter\ceil{\lceil}{\rceil}

\newcommand{\rar}{$\rightarrow$}
\newcolumntype{H}{>{\setbox0=\hbox\bgroup}c<{\egroup}@{}}

\newcommand{\gofmm}{\texttt{GOFMM}}

\newcommand{\hmatrix}{\ensuremath{\mathcal{H}}-matrix}
\newcommand{\hmatrices}{\ensuremath{\mathcal{H}}-matrices}

\newcommand{\gl}{g^{\ell}}
\newcommand{\xl}{x^{\ell}}

\newcommand{\xm}{x^{\ell-1}}
\newcommand{\xL}{x^{L}}

\newcommand{\zl}{z^{\ell}}

\newcommand{\zm}{z^{\ell-1}}
\newcommand{\zL}{z^{L}}
\newcommand{\E}{\mathbb{E}}
\newcommand{\Var}{\mathrm{Var}}

\newcommand{\hxl}{\hat{x}^{\ell}}

\newcommand{\hxm}{\hat{x}^{\ell-1}}
\newcommand{\hxL}{\hat{x}^{L}}

\newcommand{\hzl}{\hat{z}^{\ell}}

\newcommand{\hzm}{\hat{z}^{\ell-1}}
\newcommand{\hzL}{\hat{z}^{L}}

\newcommand{\hwl}{\hat{W}_{\ell}}

\newcommand{\wl}{W_{\ell}}
\newcommand{\wlt}{W^{T}_{\ell}}

\newcommand{\ml}{M^{\ell}}

\newcommand{\nomen}[1]{{{#1}}} 
\newcommand{\red}[1]{{#1}}

\newcommand{\p} {\partial}

\newcommand{\figref}[1]{Figure~\ref{#1}}
\newcommand{\tabref}[1]{Table~\ref{#1}}
\newcommand{\secref}[1]{\S\ref{#1}}

\newcommand{\lc}{{\tt l}}  
\newcommand{\rc}{{\tt r}}  

\newcommand{\zapspace}{\topsep=0pt\partopsep=0pt\itemsep=0pt\parskip=0pt}

\newcommand{\accnum}{\num[scientific-notation=true, round-mode=places, round-precision=1]}

\newcounter{magicrownumbers}
\newcommand\rownumber{\refstepcounter{magicrownumbers}\arabic{magicrownumbers}}

\begin{document}

\maketitle

\begin{abstract}
We introduce a fast algorithm for entry-wise evaluation of the Gauss-Newton Hessian (GNH) matrix for the fully-connected feed-forward neural network. The algorithm has a precomputation step and a sampling step. While it generally requires $\bigO(Nn)$ work to compute an entry (and the entire column) in the GNH matrix for a neural network with $N$ parameters and $n$ data points, our fast sampling algorithm reduces the cost to $\bigO(n+d/\epsilon^2)$ work, where $d$ is the output dimension of the network and $\epsilon$ is a prescribed accuracy (independent of $N$). One application of our algorithm is constructing the hierarchical-matrix (\hmatrix{}) approximation of the GNH matrix for solving linear systems and eigenvalue problems. It generally requires $\bigO(N^2)$ memory and $\bigO(N^3)$ work to store and factorize the GNH matrix, respectively. The \hmatrix{} approximation requires only $\bigO(N r_o)$ memory footprint and $\bigO(N r_o^2)$ work to be factorized, where $r_o \ll N$ is the maximum rank of off-diagonal blocks in the GNH matrix. We demonstrate the performance of our fast algorithm and the \hmatrix{} approximation on classification and autoencoder neural networks.
\end{abstract}



\section{Introduction} \label{s:intro} 

Consider a multilayer perceptron (MLP) with \nomen{$L$} fully connected layers and $n$ data pairs $\{(x^0_i, y_i)\}_{i=1}^n$, where $y_i$ is the label of $x^0_i$. Given input data point $x^0_i \in \mathbb{R}^{d_0}$, the output of the MLP is computed via the forward pass:
\begin{equation}
\label{e:layer}
  \nomen{\xl_i = s(\wl \, \xm_i), \quad  \ell = 1,\ldots ,L} 
\end{equation}
where $ \xl_i \in \mathbb{R}^{d_\ell},\  \wl \in \mathbb{R}^{d_\ell\times d_{\ell-1}}$ and \nomen{$s$} is a nonlinear activation function applied to every entry of the input vector. Without loss of generality, Eq.~\eqref{e:layer} does not have bias parameters. Otherwise, bias can be included in the weight matrix $\wl$, and correspondingly vector $\xl_i$ is appended with an additional homogeneous coordinate of value one.  For ease of presentation, we assume constant layer size, i.e., $d_\ell \equiv d$, for $\ell =0,1,2,\ldots,L$, so the total number of parameters is $N= d^2L$. Define the weight vector consisting of all weight parameters concatenated together as 
\[
w = [\texttt{vec}(W_{1}),\texttt{vec}(W_{2}),\ldots,\texttt{vec}(W_{L})],
\] 
where $w \in \mathbb{R}^N$ and $\texttt{vec}$ is the operator vectorizing matrices.

Given a loss function \nomen{$f(\xL_i,y_i)$}, which measures the misfit between the network output and the true label, we define 
\[
F(w) = {1 \over n} \sum_{i=1}^n f\left(\xL_i, y_i\right)
\] 
as the loss of the MLP with respect to the weight vector $w$. Note $\xL_i$ is a function of the weights $w$.

\begin{definition}[(Generalized) Gauss-Newton Hessian]
Let $Q_i = {1 \over n} \, \p^2_{\!xx} f(\xL_i, y_i) $ be the Hessian of the loss function $f(\xL_i, y_i)$ for $i=1,2,\ldots,n$, and define $Q  \in \mathbb{R}^{dn \times dn}$ as a block diagonal matrix with $Q_i$ being the $i_\mathrm{th}$ diagonal block. Let $J_i = \p_{\!w} \xL_i \in \mathbb{R}^{d \times N}$ be the Jacobian of $\xL_i$ with respect to the weights $w$ for $i=1,2,\ldots,n$, and define $J \in \mathbb{R}^{d n \times N}$ be the vertical concatenation of all $J_i$. The (generalized) Gauss-Newton Hessian (GNH) matrix $H \in \mathbb{R}^{N \times N}$ associated with the loss $F$ with respect to the weights $w$ is defined as
\begin{align} \label{e:derivatives}
 \nomen{H} = J^T Q J = \sum_{i=1}^n J_i^T  Q_i J_i.
\end{align}
\end{definition}

The GNH matrix is closely related to the Hessian matrix in that it is the Hessian matrix of a particular approximation of $F(w)$ constructed by replacing $\xL_i$ with its first-order approximation (on weights $w$)~\cite{martens16}. Importantly, the GNH matrix is always (symmetric) positive semi-definite when the loss function $f(\xL_i,y_i)$ is convex in $\xL_i$ ($Q_i$ is positive semi-definite), a useful property in many applications. In addition, for several standard choices of the loss function, the GNH matrix is mathematically equivalent to the \emph{Fisher matrix} as used in the natural gradient method.

This paper is concerned with fast entry-wise evaluation of the GNH matrix. Such an algorithmic primitive can be used in constructing approximations of the GNH matrix for solving linear systems and eigenvalue problems, which are useful for training and analyzing neural networks~\cite{byrd-e11,martens16,bottou-nocedal18,o2019inexact}, for selecting training data to minimize the inference variance~\cite{cohn94}, for estimating learning rates~\cite{lecun-bottou98}, for network pruning~\cite{hassibi-stork93}, for robust training~\cite{yao2018hessian}, for probabilistic inference~\cite{hennequin-e14}, for designing fast solvers~\cite{carmon-duchi18,triburaneni-jordan18,gower-roux-bach17} and so on.


\subsection{Previous work} 
We classify related work into two groups. One group avoids entry-wise evaluation of the GNH matrix and relies on the matrix-vector multiplication (matvec) with the Hessian or the GNH that is matrix-free~\cite{martens2010deep,martens2011learning,martens16}. For example, the matrix-free matvec can be used to construct low-rank approximations of the GNH matrix through the randomized singular value decomposition (RSVD)~\cite{halko-martinsson-tropp11}, but the numerical rank may not be small~\cite{keutzer-e17,dinh-bengio-e17}. Other examples are the following: \cite{dauphin-bengio-e14} introduces a low-rank approximation using the Lanczos algorithm to tackle saddle points; \cite{leroux-e08} maintains a low-rank approximation of the inverse of the Hessian based on rank-one updates at each optimization step; \cite{gower-roux-bach17} uses a quasi-Newton-like construction of the low-rank approximation; \cite{ye-zhang-luo17, mahoney16} study the convergence of stochastic Newton methods combined with a randomized low-rank approximation; \cite{yao2018hessian} uses a matrix-free method with only the layers near the output layer.

The other group of methods are based on evaluating or approximating entries on or close to the diagonal of the GNH matrix~\cite{lafond-bottou17}. For example, \cite{zhang-socher-e17} introduces a recursive fast algorithm to construct block-diagonal approximations. As another example, \cite{grosse-martens15,martens16} introduce the Kronecker-factored approximate curvature (K-FAC), which is based on an entry-wise approximation of the \emph{Fisher matrix} (mathematically equivalent to the GNH for some popular loss functions). The Fisher matrix is given by $ \nicefrac{1}{n} \sum_{i=1}^n \mathbb{E}_y [g_i(y) g_i(y)^T]$, where $g_i$ is the gradient evaluated for the $i_\mathrm{th}$ training point $x^0_i$, and $y$ is sampled from the network's predictive distribution $\propto \exp(-f(\xL_i,y))$. In practice, an extra step of block-diagonal or block-tridiagonal approximation is used for fast inversion purpose. The method has been tested within optimization frameworks on modern supercomputers and has been shown to perform well~\cite{osawa-yokota-satoshi-e18}. However, the sampling in the K-FAC algorithm converges slowly, and block-diagonal approximations do not account for off-diagonal information.


\subsection{Contributions}
In this paper, we introduce a fast algorithm for entry-wise evaluation of the GNH matrix  $H$, i.e., computing
\[
H_{km} = e_k^T \, H \, e_m,
\] 
where $e_k$ and $e_m$ are two canonical bases for $k,m=1,2,\ldots,N$. With the fast evaluation, we propose the hierarchical-matrix (\hmatrix{}) approximation~\cite{bebendorf08,hackbusch15} of the GNH matrix for the MLP network, which has applications in autoencoders, long-short memory networks, and is often used to study the potential of second-order training methods. Notice if the matrix-free matvec is used to evaluate $H_{km}$, the computational cost would be $\bigO(Nn)$.

Our fast algorithm includes a precomputation step and a sampling step, which reduces the cost to $\bigO(n+d)$ work (independent of $N$), where $d$ is the output dimension of the network. To illustrate the idea, suppose the network employs the mean squared loss, i.e., $f(\xL_i,y_i)={1\over 2}\|\xL_i-y_i\|^2$, and therefore, the GNH matrix is $H= {1 \over n} J J^T$, where $J \in \mathbb{R}^{d n \times N}$ is the Jacobian of the network output with respect to the weights. Then $H_{km} = {1 \over n} (J e_k)^T  \, (Je_m)$, and only columns in the Jacobian are required to be computed. Our precomputation algorithm exploits the structure of a feed-forward neural network, where the gradient is back propagated layer by layer, so the intermediate results effectively form a compressed format of the Jacobian with $\bigO(Nn)$ memory. As a result, every column can be retrieved in only $\bigO(nd)$ time (note every column has $\bigO(nd)$ entries).

To accelerate the computation of $H_{km}$, we introduce a fast Monte Carlo sampling algorithm. Let $v_k(i)$ denote the sub-vector in the Jacobian's $k_\mathrm{th}$ column corresponding to the $i_\mathrm{th}$ data point, and therefore, $H_{km} =  {1 \over n} \sum_{i=1}^n v_k(i)^T v_m(i)$. In the sampling, we draw $c$ (independent of $n$) independent samples $t_1, t_1, \ldots, t_c$ from $\{1,2,\ldots,n\}$ with a carefully designed probability distribution $P_{km}$ and compute an estimator 
\[
\tilde{H}_{km} = \frac{1}{nc}\sum_{j=1}^{c} \frac{v_k( t_j )^T v_m( t_j )}{P_{km}(t_j)}.
\] 
We prove $| H_{km} - \tilde{H}_{km}| = \bigO(1/\sqrt{c})$ with high probability. Note it requires only $\bigO(n+dc)$ work to compute $\tilde{H}_{km}$ as an approximation, where $d$ is the output dimension of the network.

With the fast evaluation algorithm, we are able to take advantage of the existing \gofmm{} method~\cite{chenhan-biros-e17,yu-reiz-biros18,gofmm-home-page} to construct the \hmatrix{} approximation of the GNH matrix through evaluating $\bigOT(N)$ entries in the matrix. The \hmatrix{} approximation is a multilevel scheme that stores diagonal blocks and employs low-rank approximations for off-diagonal blocks in the input matrix. So previous work on the (global) low-rank approximation and the block-diagonal approximation can be viewed as the two extremes in the spectrum of our \hmatrix{} approximation, which effectively works for a broader range of problems. \hmatrices{} are algebraic generalizations of the well-known fast $n$-body calculation algorithms~\cite{barnes-hut-86,greengard94} in computational physics, and they have been applied to kernel methods in machine learning~\cite{lee-gray08,march-xiao-yu-biros-sisc16}. An \hmatrix{} can be formulated as 
\begin{equation}\label{e:hmat}
\nomen{H = D + S + U V^T}
\end{equation}
where $U$ and $V$ are tall-and-skinny matrices, $S$ is a block-sparse matrix, and $D$ is a block-diagonal matrix with the blocks being either smaller \hmatrices{} at the next level or dense blocks at the last level. \figref{f:hmat} shows the structure of a low-rank matrix and the hierarchically low-rank structure of \hmatrices{}.

\begin{figure}[htbp]
\begin{center}
\includegraphics[width=0.69\textwidth]{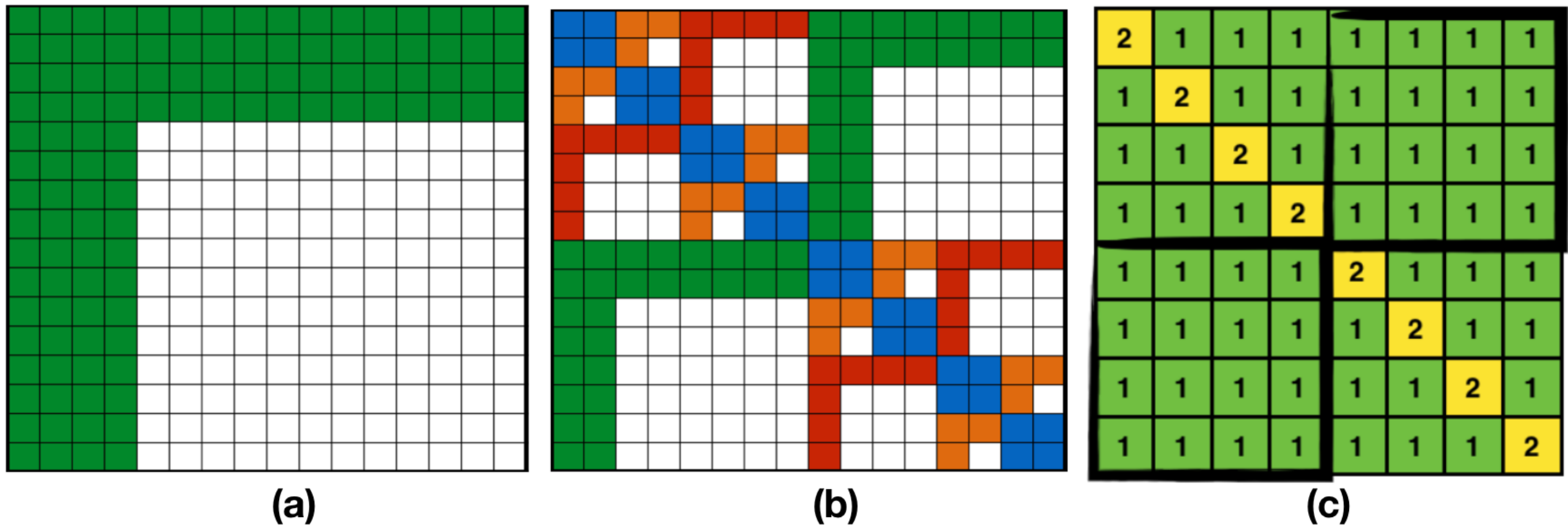}
\caption{\em {\bf (a)} A low-rank matrix $H=UV^T$, where $U$ and $V$ are tall-and-skinny matrices; {\bf (b)} a three-level \hmatrix{}, where blue represents dense diagonal blocks, and green, red and orange represent off-diagonal low-rank blocks at level 1, 2 and 3, respectively; {\bf (c)} ranks of diagonal and off-diagonal blocks in an \hmatrix{}, where every block has size 2-by-2.}
\label{f:hmat}
\end{center}
\end{figure}

Given an \hmatrix{} approximation, the memory footprint is $\bigOT(N r_o)$\footnote{Generally speaking, there may be a $\log(N)$ or $\log^2(N)$ prefactor, as for other complexity results related to \hmatrix{} approximations. But here we focus on the case without such prefactors.}, where $N$ is the matrix size or the number of weights in a network and $r_o$ is the maximum off-diagonal rank. Compared to the $\bigO(N^2)$ storage for the entire matrix, an \hmatrix{} approximation leads to significant memory savings. Once constructed, an \hmatrix{} can be factorized with only $\bigOT(N r_o^2)$ work, and there exists an entire class of well-established numerical techniques~\cite{martinsson2005fast,xia2010fast,ho2013hierarchical,ghysels-li-e16,aminfar2016fast,chen2018distributed,takahashi2019parallelization}. The factorization can be applied to a vector with $\bigOT(N r_o)$ work and be used as either a fast direct solver or a preconditioner depending on the approximation accuracy.

To summarize, our work makes the following two major contributions:
\begin{itemize}
\item a fast algorithm that requires $\bigO(Nn)$ storage and requires $\bigO(n+d/\epsilon^2)$ work to evaluate an arbitrary entry in the GNH matrix, where $N$ and $d$ are the number of parameters and the output dimension of the MLP, respectively, $n$ is the data size, and $\epsilon$ is a prescribed accuracy.
\item a framework to construct the \hmatrix{} approximation of the GNH matrix, an analysis of the approximation accuracy and the cost, as well as comparison with the RSVD and the K-FAC methods. 
\end{itemize}

\paragraph{Outline}  
In \secref{s:setup} we review some background material. In \secref{s:algorithms} we present our fast algorithm for evaluating entries in the GNH matrix. In \secref{s:gofmm} we show how to construct the \hmatrix{} approximation of the GNH matrix. In \secref{s:results} we show numerical results, and in \secref{s:conclusion} we conclude with further extensions.
Throughout this paper, we use $\|\cdot\|$ to denote the vector/matrix 2-norm and  $\|\cdot\|_F$ to denote the matrix Frobenius norm.


\section{Background} \label{s:setup} 

In this section, we review the importance of the GNH matrix and the associated computational challenge. The GNH matrix is useful in training and analyzing neural networks, selecting training data, estimating learning rate, and so on. Here we focus on its use in second-order optimization to show the challenge that is common in other applications. 

\subsection{Neural network training}

In an MLP, the weight vector $w$ is obtained via solving the following constrained optimization problem (regularization on $w$ could be added):
\begin{equation}\label{e:opt}
  \begin{split}
     \min F(w) = \min {1\over n} \sum_{i=1}^n f\left(\xL_i, y_i\right) \\
    \quad \mbox{subject to Eq.~\eqref{e:layer}}.
  \end{split}
\end{equation}
Recall that $f$ is the loss function, $\xL_i$ is the network output corresponding to input $x^0_i$, which has label $y_i$.

To solve for $w$ in problem~\eqref{e:opt}, a second-order optimization method solves a sequence of local quadratic approximations of $F(w)$, which requires solving the following linear systems repeatedly:
\begin{equation} \label{e:system}
H p = -g,
\end{equation}
where $H$ is the \emph{curvature matrix} (the Hessian of $F(w)$ in the standard Newton's method), $g = \p_{\!w}  F$ is the gradient, and $p$ is the update direction. Generally speaking, second-order optimization methods are  highly concurrent and could require much less number of iterations to converge than first-order methods, which imply potentially significant speedup on modern distributed computing platforms. 

In the Gauss-Newton method, a popular second-order solver, the GNH matrix is employed (with a small regularization) as the curvature matrix in Eq.~\eqref{e:system}, which can be solved using the Conjugate Gradient method. Since the GNH is mathematically equivalent to the Fisher matrix for several standard choices of the loss, and then the solution of Eq.~\eqref{e:system} becomes the \emph{natural gradient}, a efficient steepest descent direction in the space of probability distribution with an appropriately defined distance measure~\cite{martens2014new}.


\subsection{Back-propagation \& matrix-free matvec}

\begin{table}[tbhp]
{\footnotesize
  \caption{\red{Gradient evaluation and matrix-free matvec with the GNH matrix. Step (a) of gradient evaluation is the forward pass in Eq.~\eqref{e:layer}, and step (b) and (c) are the well-known back-propagation. In the matvec, step (a) is known as the linearized forward ($\hat{x}_i^0 = 0$), which computes $J \hat{w}$. Notations: 
$\ml_i = \mathrm{diag}\left(\dot{s}(\wl \, \xm_i)\right)$, where $\dot{s}$ stands for the derivative;
  $\gl$ is the gradient for $\wl$; 
  $\hat{w} = [\texttt{vec}(\hat{W}_{1}),\ldots,\texttt{vec}(\hat{W}_{L})]$ is the input of the matvec;
  $(H\hat{w})^{\ell}$ refers to the $\ell_\mathrm{th}$ block of the output.}
  }\label{tab:cg}
\begin{center}
  \begin{tabular}{|c|c|} \hline
  Evaluate gradient $g$  & Matvec with GNH: $H\hat{w} = J^T Q J \hat{w}$ (Definition 1.1) \\ 
  & \\
  \arraycolsep=1pt\def\arraystretch{1.3}
 $ \begin{array}{lll}
\mbox{(a)}\  \xl_i &= s\left(\wl \, \xm_i\right) & \forall i, \ell   \\
\mbox{(b)}\ \zL_i &= \red{\p_{\!x} f(\xL_i, y_i),} &  \forall i\\
\mbox{(c)}\ \zm_i &= \wlt\ml_i \zl_i& \forall i,  \ell  \\
\mbox{(d)}\ \nomen{\gl}  &= \sum_{i=1}^n \left(\ml_i \zl_i\right)\, (x_i^{\ell-1})^T & \forall \ell\\
  \end{array}$
   &
   \arraycolsep=1pt\def\arraystretch{1.3}
  $\begin{array}{llll} \label{e:matvec}
\mbox{(a)}\  \hxl_i &= \ml_i \left( \wl \hxm_i + \hwl \xm_i \right) &  \forall i, \ell  \\
\mbox{(b)}\ \hzL_i &= Q_i \hxL_i &\forall i  \\
\mbox{(c)}\  \hzm_i &= \wlt\ml_i \hzl_i &\forall i, \ell \\
\mbox{(d)}\ \nomen{(H\hat{w})^{\ell}} &= \sum_{i=1}^n \left(\ml_i \hzl_i\right)\, (x_i^{\ell-1})^T &\forall \ell \\
  \end{array}$
   \\ \hline
  \end{tabular}
\end{center}
}
\end{table}

\cref{tab:cg} shows the algorithm known as back-propagation for evaluating the gradient $g = \p_{\!w} F$ and the matrix-free matvec with the GNH matrix, both of which have complexity $\bigO(N n)$. Both algorithms can be derived by introducing Lagrange multipliers $\zl_i$ and $\hzl_i$ for the corresponding weights $\wl$ at every layer~\cite{gill-murray-wright-81,goodfellow16}. Note a direct matvec with the full GNH matrix would require $\bigO(N^2)$ work, not even mentioning the amount of work to compute the entire matrix.

Based on the two basic ingredients, iterative solvers such as Krylov methods can be used to solve Eq.~\eqref{e:system} as in Hessian-free methods~\cite{martens2010deep,martens2011learning}. However, the iteration count for convergence can grow rapidly in the presence of ill-conditioning, in which case fast solvers or preconditioners for Eq.~\eqref{e:system} are necessary~\cite{axelsson-94,knoll-keyes-04,martens16}.

\section{Fast computation of entries in GNH} \label{s:algorithms} 

This section presents a precomputation algorithm and a fast Monte Carlo algorithm for fast computation of \emph{arbitrary} entries in the GNH matrix of an MLP network.


\paragraph{A naive method} Consider a GNH matrix $H \in \mathbb{R}^{N \times N}$, where an entry $H_{km}$ can be written as 
\begin{equation} \label{e:hkm}
H_{km} = e_k^T \, H \, e_m, 
\end{equation}
where $e_k$ and $e_m$ are the $k_\mathrm{th}$ and the $m_\mathrm{th}$ columns of the $N$-dimensional identity matrix. We can take advantage of the matrix-free matvec with the GNH matrix in \cref{tab:cg} to compute $H_{km} = e_k^T (H e_m)$, which costs the same as one pass of forward propagation plus one pass of backward propagation, i.e., $\bigO(N n) = \bigO(d^2 L n)$ work. 

In the following, we introduce a precomputation algorithm that reduces the cost of evaluating an entry in the GHN to $\bigO(d n)$ work with $\bigO(N n)$ memory, and a fast Monte Carlo algorithm that further reduces the cost to $\bigO(n+d/\epsilon^2)$ work, where $\epsilon$ is a prescribed accuracy that does not depend on $n$ nor $N$. 

\subsection{Precomputation algorithm} \label{s:precompute}
The motivation of our precomputation algorithm is to exploit the sparsity of $e_k$ and $e_m$ plus the symmetry of $H$ in Eq.~\eqref{e:hkm}. Recall the definition of $H$ in Eq.~\eqref{e:derivatives}, and let $Q_i = R_i^TR_i$ be a symmetric factorization, which can be computed via, e.g., the eigen-decomposition or the LDLT factorization with pivoting. We have
\begin{equation}\label{e:jachess}
\arraycolsep=1pt\def\arraystretch{1.3}
\begin{array}{ll}
H_{km}
&= e_k^T \big(\sum_{i=1}^n J_i^T  R_i^T R_i J_i \big) e_m 
= \sum_{i=1}^n (R_i J_i e_k)^T (R_i J_i e_m)  \\
&:= \sum_{i=1}^n v_k(i)^T v_m(i)
\end{array}
\end{equation}
where $v_k(i)$ and $v_m(i)$ are two $d$-dimensional vectors:
\begin{equation} \label{e:vkm}
v_k(i) = R_i J_i e_k, \quad v_m(i) = R_i J_i e_m,
\end{equation}
for $k,m=1,2,\ldots,N$ and $i=1,2,\ldots,n$. We state the following theorem and present the precomputation algorithm in the proof.

\begin{theorem}\label{th:th1} For an MLP network that has $L$ fully connected layers with constant layer size $d$ ($d$-by-$d$ weight matrices), every entry $H_{km}$ in the GNH matrix can be computed in $\bigO(d n)$ time with a precomputation that requires $\bigO(n N)$ storage and $\bigO(d n N)$ work.
\end{theorem}

\begin{proof}
\red{We first compute $x_i^{\ell-1}$ and $M^\ell_i = \mathrm{diag}(\dot{s} (W_\ell x_i^{\ell-1}))$ via the forward pass, i.e., step (a) of gradient evaluation in \cref{tab:cg},} and then we precompute and store
\begin{equation} \label{e:c}
C_i^{\ell} = R_i M^L_i W_L M_i^{L-1}W_{L-1}\cdots M_i^{\ell}, \quad i=1,2,\ldots, n; \, \ell = 1, 2, \ldots, L.
\end{equation}
Since every $C_i^{\ell}$ is a $d \times d$ matrix, the total storage cost is $\bigO(d^2 \, n L)= \bigO(n N)$, where $N = d^2L$ is the total number of weights. In addition, notice the relation that $C_i^{\ell-1} = C_i^{\ell} \left(W_{\ell} M_i^{\ell-1} \right)$, so they can be computed from $\ell = L$ to $\ell = 1$ iteratively, which requires $\bigO(d^3 L n) = \bigO(d n N)$ work in total. Note that the forward pass costs $\bigO(n N)$ work, and that computing the symmetric factorizations for $Q_i$ cost $\bigO(d^3 n)$, which is negligible compared to other parts of the computation. 


To complete the proof we show how to compute $v_k(i)$ as defined in Eq.~\eqref{e:vkm} with $\bigO(d)$ work. \red{Below we use the same notations as in \cref{tab:cg}, and $e_k$ is the input vector of the matvec corresponding to $\hat{w} = [\texttt{vec}(\hat{W}_{1}),\ldots,\texttt{vec}(\hat{W}_{L})]$ in \cref{tab:cg}. Recall step (a) of the matrix-free matvec (linearized forward) with the GNH in \cref{tab:cg}, and we evaluate $J_i e_k$ as follows.}
\begin{enumerate}
  \item 
  Let $\tau = \ceil*{k/{d^2}}$, $\mu = k \bmod{d}$, and $\nu= \ceil*{(k \bmod{d^2})/{d}}$. Since $e_k$ has only one nonzero entry, $\hxl_i = 0$ for $\ell=1, 2, \ldots, \tau-1$ because $\hwl$ are all zeros except for $\ell = \tau$. The matrix $\hat{W}_{\tau}$ has only one nonzero at position $(\mu, \nu)$ (column-major ordering) as the following:
\[
\bordermatrix{& &\nu& \cr
                & &  \vdots  &  \cr
                \mu& \ldots  &  1 & \ldots \cr
                &  & \vdots &  \cr}
= \hat{W}_{\tau}.
\] 
  \item 
  Following step (a) of the matvec in \cref{tab:cg}, we have $\hat{x}^{\tau}_i = M_i^{\tau} \hat{W}_{\tau} x^{\tau-1}_i$ at layer $\tau$. Denote $a^{\tau}_i =  \hat{W}_{\tau} x^{\tau - 1}_i$, and we have 
  \begin{align*}
  \hat{x}^{\tau}_i &= M_i^{\tau}  a^{\tau}_i \\
  \hat{x}^{\tau+1}_i &= M^{\tau+1}_i  W_{\tau+1} \hat{x}^{\tau}_i & \text{(since $\hat{W}_{\tau+1}=0$)} \\
  &= M^{\tau+1}_i  W_{\tau+1} M_i^{\tau}  a^{\tau}_i \\
   & \ldots  \\
   \hat{x}^L_i &= M^L_i W_L M_i^{L-1}W_{L-1}\cdots M_i^{\tau} a^{\tau}_i
  \end{align*}
  \item 
  Notice that the only nonzero entry in $a^{\tau}_i$ is the $\mu_\mathrm{th}$ element, which equals to the $\nu_\mathrm{th}$ element in ${x}^{\tau-1}_i$. Therefore, 
        \begin{equation} \label{e:vki}
        \nomen{v_k(i) = R_i J_i e_k = R_i \hat{x}^L_i = C_i^{\tau} a^\tau_i},
        \end{equation}
        where $C_i^{\tau} a^\tau_i$ should be interpreted as a scaling of the $\mu_\mathrm{th}$ column in $C_i^{\tau}$ by the $\nu_\mathrm{th}$ element in ${x}^{\tau-1}_i$, which costs $\bigO(d)$ work.
\end{enumerate}
\end{proof}

\subsection{Fast Monte Carlo algorithm} \label{s:sampling}

Recall Eq.~\eqref{e:jachess}, which sums over a large number of data points, and the idea is to sample a subset with judiciously chosen probability distribution and scale the (partial) sum appropriately to approximate $H_{km}$. It is important to note that the computation of the probabilities is fast based on the previous precomputation. The fast sampling algorithm is given in~\cref{alg:sampling}.

\begin{algorithm}[htbp]
\caption{Fast Monte Carlo Algorithm}
\label{alg:sampling}
\begin{algorithmic}[1]
\STATE{\textbf{Input:} $\|v_k(i)\|$ and $\|v_m(i)\|$ for $i=1,2,\ldots,n$.}
\STATE{Compute sampling probabilities for $t=1,2,\ldots,n$:
\begin{equation} \label{e:probability}
\nomen{P_{km}(t) = \frac{\|v_k(t)\| \, \|v_m(t)\|}{\sum_{j=1}^n \|v_k(j)\| \, \|v_m(j)\|}}.
\end{equation}
}
\STATE{Draw $c$ independent random samples $t_j$ from $\{1,2,\ldots,n\}$ with replacement.}
\STATE{\textbf{Output:}
\begin{equation} \label{e:estimator}
\tilde{H}_{km} = \frac{1}{c}\sum_{j=1}^{c} \frac{v_k( t_j )^T v_m( t_j )}{P_{km}(t_j)}.
\end{equation}
}
\end{algorithmic}
\end{algorithm}

Define $v_k = [v_k(1),\dots, v_k(n)]$ and $v_m = [v_m(1),\dots, v_m(n)]$ as two vectors in $\mathbb{R}^{dn}$, and Eq.~\eqref{e:jachess} can be written as the inner product of the two vectors: 
\[
H_{km} = v_k^T v_m.
\]
The following theorem shows that our sampling algorithm returns a good estimator of $H_{km}$, where the error is measured using $\|v_k\| \|v_m\|$, an upper bound on $|H_{km}|$.

\begin{theorem}[Sampling error] \label{th:th2} 
Consider an MLP network that has $L$ fully connected layers with constant layer size $d$ ($d$-by-$d$ weight matrices). For every entry $H_{km}$ in the GNH matrix, \cref{alg:sampling} returns an estimator $\tilde{H}_{km}$ that 
\begin{itemize}
\item 
is an unbiased estimator of $H_{km}$, i.e., $\E[\tilde{H}_{km}] = H_{km}$.
\item 
its variance or mean squared error (MSE) satisfies
\begin{equation} \label{e:mse}
\Var[\tilde{H}_{km}] = \E \left[ |{H}_{km} - \tilde{H}_{km}|^2 \right] \le \frac{1}{c} \|v_k\|^2 \|v_m\|^2
\end{equation}
where $c$ is the number of random samples.
\item
with probability at least $1-\delta$, where $\delta\in(0,1)$, its absolute error satisfies
\begin{equation} \label{e:error}
|H_{km} - \tilde{H}_{km}| \leq \frac{\eta}{\sqrt{c}} \ \|v_k\| \|v_m\|
\end{equation}
where $\eta = 1+\sqrt{8 \log(1/\delta)}$ and $c$ is the number of random samples.
\end{itemize}
\end{theorem}

\begin{proof} 
Our proof consists of the following three parts.
\paragraph{Unbiased estimator} Define a random variable 
\[
X_t = \frac{v_k( t )^T v_m( t )}{P(t)}
\]
where $t$ is a random sample from $\{1,2,\ldots,n\}$ with probability distribution $P(t)$ as defined in Eq.~\eqref{e:probability}. \red{Observe that $\tilde{H}_{km}$ is the mean of $c$ independent identically distributed variables ($X_{t_1}, X_{t_2}, \ldots, X_{t_c}$), and thus}
\[
\E[\tilde{H}_{km}] = \E[X_t] = \sum_{t=1}^n \frac{v_k( t )^T v_m( t )}{P(t)} P(t) = H_{km}.
\]

\paragraph{Variance/MSE error}  The variance or MSE error of the estimator is the following:
\begin{align*}
& \E \left[ |{H}_{km} - \tilde{H}_{km}|^2 \right] =  \Var [ \tilde{H}_{km}] = \frac{1}{c} \Var[X_t]  \\
=& \frac{1}{c}  \big(\E[X_t^2] - \E^2[X_t] \big)  \\
=& \frac{1}{c}  \sum_{t=1}^n \bigg( \frac{v_k( t )^T v_m( t )}{P(t)}\bigg)^2 P(t) - \frac{H_{km}^2}{c}  \\
\le& \sum_{t=1}^n  \frac{ \big(v_k( t )^T v_m( t ) \big)^2}{c \, P(t)}   & \mbox{(Drop the last term)} \\
\le& \sum_{t=1}^n \frac{\|v_k( t )\|^2 \|v_m( t )\|^2}{c \, P(t)}  & \mbox{(Cauchy-Schwarz)} \\
=& \frac{1}{c} \bigg( \sum_{t=1}^n \|v_k( t )\| \|v_m( t )\| \bigg)^2  & \mbox{(Eq.~\eqref{e:probability})}  \\
\le& \frac{1}{c} \bigg( \sum_{t=1}^n \|v_k( t )\|^2 \bigg) \bigg( \sum_{t=1}^n \|v_m( t )\|^2 \bigg) & \mbox{(Cauchy-Schwarz)} \\
=& \frac{1}{c} \|v_k\|^2 \|v_m\|^2.
\end{align*}
Notice that with Jensen's inequality, we also obtain a bound of the absolute error in expectation: 
\begin{equation} \label{e:error2}
\E \left[ |{H}_{km} - \tilde{H}_{km}| \right] \le  \frac{1}{\sqrt{c}} \|v_k\| \|v_m\|.
\end{equation}

\paragraph{Concentration result} We will use the McDiarmid's (a.k.a., Hoeffding-Azuma or Bounded Differences) inequality to obtain Eq.~\eqref{e:error}. See the conditions for the inequality in \secref{s:app}. Define function 
\[
F(t_1, t_2, \ldots, t_c) = |{H}_{km} - \tilde{H}_{km}|,
\] 
where $t_1, t_2, \ldots, t_c$ are random samples, and we show that changing one sample $t_i$ at a time does not affect $F$ too much. Consider changing a sample $t_i$ to $t_i'$ while keeping others the same. The new estimator $\hat{H}_{km}$ differs from $\tilde{H}_{km}$ by only one term. Thus,
\begin{align*}
|\tilde{H}_{km} - \hat{H}_{km}| 
&= \left | \frac{v_k( t_i )^T v_m( t_i )}{c \ P(t_i)} - \frac{v_k( t_i' )^T v_m( t_i' )}{c \ P(t_i')} \right | \\
&\le \left | \frac{v_k( t_i )^T v_m( t_i )}{c \ P(t_i)} \right | + \left | \frac{v_k( t_i' )^T v_m( t_i' )}{c \ P(t_i')} \right | \\
&\le \frac{\|v_k( t_i )\| \|v_m( t_i )\|}{c \ P(t_i)} + \frac{\|v_k( t_i' )\| \|v_m( t_i' )\|}{c \ P(t_i')} \\
&= \frac{2}{c} \sum_{j=1}^n \|v_k(j)\| \, \|v_m(j)\| \\
&\le \frac{2}{c} \| v_k \| \| v_m \|.
\end{align*}
where we have used Cauchy-Schwarz inequality twice. Then, define $\Delta = \frac{2}{c} \| v_k \| \| v_m \|$; using the triangle inequality we see 
\[
|F(\ldots, t_i, \ldots) - F(\ldots, t_i', \ldots) | \le \Delta.
\]
Finally, let $\gamma = \sqrt{2c \log (1/ \delta)} \, \Delta$, and we use the McDiarmid's inequality to obtain Eq.~\eqref{e:error} as follows
\begin{align*}
& \mbox{Pr} \left[ |H_{km} - \tilde{H}_{km}| \ge \frac{\eta}{\sqrt{c}} \ \|v_k\| \|v_m\| \right] \\
=& \mbox{Pr} \left[ |{H}_{km} - \tilde{H}_{km}| \ge \frac{1}{\sqrt{c}} \|v_k\| \|v_m\| + \gamma \right] \\
\le & \mbox{Pr}  \left[ F - \E [F] \ge \gamma \right]  & \mbox{(Eq.~\eqref{e:error2})}\\
\le & \mbox{exp} \left(- \frac{\gamma^2}{2c \Delta^2} \right) = \delta & \mbox{(McDiarmid's inequality)}.
\end{align*}
\end{proof}

\begin{remark}
The error $\epsilon$ in the approximation of $H_{km}$ depends on only the number of random samples $c$ (but not $n$) and can be made arbitrarily small as needed. In particular, if $c \ge 1/\epsilon^2$, we have
\[
\Var[\tilde{H}_{km}] = \E \left[ |{H}_{km} - \tilde{H}_{km}|^2 \right] \le \epsilon \ \|v_k\|^2 \|v_m\|^2
\] 
and if $c \ge \eta^2 /\epsilon^2$, then with probability at least $1-\delta$, where $\delta\in(0,1)$ 
\[
|H_{km} - \tilde{H}_{km}| \leq \epsilon \ \|v_k\| \|v_m\|.
\]
Furthermore, the error of the entire matrix in the Frobenius norm is
\begin{align*}
\| H - \tilde{H} \|_{{F}} \le & \epsilon \sqrt{ \sum_k \sum_m \|v_k\|^2 \|v_m\|^2 } = \epsilon \sum_k \|v_k\|^2 \\
& \red{\stackrel{\cref{e:jachess}}{=} \epsilon \sum_k {H}_{kk}} = \epsilon \,\mbox{trace}(H) \le \epsilon \ \sqrt{N} \ \| H \|_{{F}}.
\end{align*}
\end{remark}

\begin{remark}\label{rmk}
The estimator $\tilde{H}_{km}$ is exact using at most \emph{one} sample when $k=m$. The (trivial) case $H_{kk} = 0$ is implied by the situation that $v_{kk}(i) = 0$ for all $i$; otherwise, we have $H_{kk} = \|v_k\|^2$, and the sampling probability becomes
\[
\nomen{P_{kk}(t) = \frac{\|v_k(t)\|^2}{\sum_{j=1}^n \|v_k(j)\|^2} = \frac{\|v_k(t)\|^2}{\|v_k\|^2}}.
\]
Therefore, $\tilde{H}_{kk} = \|v_k(t)\|^2 / P_{kk}(t) = H_{kk}$ with any random sample $t$.
\end{remark}

\begin{theorem} [Computational cost of sampling] \label{th:th3} 
Given the precomputation in \cref{th:th1}, it requires $\bigO(n N)$ work to compute $\|v_k(i)\|$ for all $i$ and $k$ as the input of \cref{alg:sampling}, and it requires $\bigO(n + d/\epsilon^2)$ work to compute every estimator, where $\epsilon$ is a prescribed accuracy that does \emph{not} depend on $n$.
\end{theorem}

\begin{proof}
Recall Eq.~\eqref{e:vki} that $\|v_k(i)\|$ is proportional to the norm of a column in $C_i^\ell$. Since every $C_i^\ell$ is a $d$-by-$d$ matrix, computing all the norms requires $\bigO(d^2  n L) = \bigO(n N)$ work. Once all $\|v_k(i)\|$ have been computed, the sampling probabilities in Eq.~\eqref{e:probability} and the estimator in Eq.~\eqref{e:estimator} requires $\bigO(n)$ and $\bigO(d/ \epsilon^2)$ work, respectively.
\end{proof}

\section{\hmatrix{} approximation} \label{s:gofmm} 

This section introduces the \hmatrix{} approximation of the GNH matrix for the MLP. While the low-rank and the block-diagonal approximations focus on the global and the local structure of the problem, respectively, the \hmatrix{} approximation handles both as they may be equally important. 

\subsection{Overall algorithm}
Here we take advantage of the \gofmm{} method~\cite{chenhan-biros-e17,yu-reiz-biros18,gofmm-home-page}, which evaluates $\bigO(N)$ entries in a symmetric positive definite (SPD) matrix $H \in \mathbb{R}^{N \times N}$ to construct the \hmatrix{} approximation $H_{\gofmm}$ such that
\[
\|H-H_{\gofmm}\|_F \le \epsilon \, \|H\|_F,
\] 
where $\epsilon$ is a prescribed tolerance.

Since \gofmm{} requires only entry-wise evaluation of the input matrix, we apply it with our fast evaluation algorithm to the regularized GNH matrix (note the GNH matrix is symmetric positive semi-definite, so we always add a small regularization of $\lambda$ times the identity matrix, where $\lambda^2$ is the unit roundoff). The overall algorithm that computes the \hmatrix{} approximation (and approximate factorization) of the GNH matrix using the \gofmm{} method is shown in \cref{alg:gofmm}.

\begin{algorithm}[htbp]
   \caption{Compute \hmatrix{} approximation of GNH with \gofmm{}}
   \label{alg:gofmm}
\begin{algorithmic}[1]
   \REQUIRE training data $\{x_i^0\}_{i=1}^n$, weights in the neural network $w \in \mathbb{R}^N$
   \ENSURE approximation of the GNH and its factorization 
   \STATE Compute $M^\ell_i$ with forward propagation. (step (a) of gradient evaluation in \cref{tab:cg})
   \STATE Compute $C^\ell_i$ in Eq.~\eqref{e:c}. (\cref{th:th1}: $\bigO(N n d)$ work and $\bigO(N n)$ storage)
   \STATE Compute $\|v_k(i)\|$ in Eq.~\eqref{e:vki}. (\cref{th:th3}: $\bigO(N n)$ work and $\bigO(N n)$ storage)
   \STATE Apply \gofmm{} and evaluate entries in the GNH matrix through \cref{alg:sampling}. (\cref{th:th3}: $\bigO(n+d/\epsilon^2)$ work/entry)
\end{algorithmic}
\end{algorithm}

The error analysis of \cref{alg:gofmm} is the following. Let $\tilde{H}_{\lambda} = \tilde{H} + \lambda I$ be computed by \cref{alg:sampling} and $\lambda>0$ is a regularization, and $\tilde{H}_{\gofmm{}}$ be the approximation of $\tilde{H}_{\lambda}$ computed by \gofmm{}. Then the error between the output $\tilde{H}_{\gofmm{}}$ from \cref{alg:gofmm} and the (regularized) GNH matrix $H_{\lambda} = H + \lambda I$ is the following \red{(using the triangular equality)}
\begin{align*}
\| H_{\lambda} -  \tilde{H}_{\gofmm{}} \|_F 
= \| H_{\lambda} - \tilde{H}_{\lambda} + \tilde{H}_{\lambda} - \tilde{H}_{\gofmm{}} \|_F 
 \le \| H- \tilde{H} \|_F + \| \tilde{H}_{\lambda} -  \tilde{H}_{\gofmm{}} \|_F,
\end{align*}
where the first term is the sampling error from \cref{alg:sampling} and the second term is the \gofmm{} approximation error. For simplicity, we drop the regularization parameter for the rest of this paper.

\subsection{\gofmm{} overview}

Given an SPD matrix $H$, the \gofmm{} takes two steps to construct the \hmatrix{} approximation as follows. First of all, a permutation matrix $P$ is computed to reorder the original matrix, which often corresponds to a hierarchical domain decomposition for applications in two- or three-dimensional physical spaces. The recursive domain partitioning is often associated with a tree data structure $\cal{T}$. Unlike methods targeting applications in physical spaces, the \gofmm{} does not require the use of geometric information (thus its name``geometry-oblivious fast multipole method"), which does not exist for neural networks. Instead of relying on geometric information, the \gofmm{} exploits the algebraic distance measure that is implicitly defined by the input matrix $H$. As a matter of fact, any SPD matrix $H \in\mathbb{R}^{N\times N}$ is the \emph{Gram matrix} of $N$ \emph{unknown Gram vectors} $\{\phi_i\}_{i=1}^N$~\cite{hofmann2008kernel}. Therefore, the distance between two row/column indices $i$ and $j$ can be defined as 
\begin{equation} \label{e:angle}
d_{ij} = \sin^2 \left( \angle(\phi_i, \phi_j) \right) = 1 - H^{2}_{ij}/(H_{ii}H_{jj}).
\end{equation}
or
\[
d_{ij} = \|\phi_i - \phi_j\| = \sqrt{H_{ii} - 2H_{ij} + H_{jj}},
\]
\red{We refer interested readers to~\cite{chenhan-biros-e17} for the discussion and comparison of different distance metrics.} With either definition, the \gofmm{} is able to construct the permutation $P$ and a balanced binary tree $\cal{T}$.

The second step is to approximate the reordered matrix $P^THP$ by
\begin{equation*}
\label{e:partitioning}
{H}_{\gofmm{}} =
\begin{bmatrix}
{H}_{\lc\lc} & 0 \\ 
0 & {H}_{\rc\rc} \\ 
\end{bmatrix} + 
\begin{bmatrix} 
0 & S_{\lc\rc} \\ 
S_{\rc\lc} & 0 \\ 
\end{bmatrix}+
\begin{bmatrix} 
0 & U_{\lc\rc}V_{\lc\rc}^T \\ 
U_{\rc\lc}V_{\rc\lc}^T & 0 \\ 
\end{bmatrix}, 
\end{equation*} 
where ${H}_{\lc\lc}$ and ${H}_{\rc\rc}$ are two diagonal blocks that have the same structure as ${H}_{\gofmm{}}$ unless their sizes are small enough to be treated as dense blocks, which occurs at the leaf level of the tree $\cal{T}$; $S_{\lc\rc}$ and $S_{\rc\lc}$ are block-sparse matrices, and $U_{\lc\rc}V_{\lc\rc}^T$ and $U_{\rc\lc}V_{\rc\lc}^T$ are low-rank approximations of the remaining off-diagonal blocks in $H$. These bases are computed recursively with a post-order traversal of $\cal{T}$ using the interpolative decomposition~\cite{halko-martinsson-tropp11} and a nearest neighbor-based fast sampling scheme. There is a trade-off here: while the so-called weak-admissibility criteria sets $S_{\lc\rc}$ and $S_{\rc\lc}$ to zero and obtains relatively large ranks, the so-called strong-admissibility criteria selects $S_{\lc\rc}$ and $S_{\rc\lc}$ to be certain subblocks in $H$ corresponding to a few nearest neighbors/indices of every leaf node in $\cal{T}$ and achieves smaller (usually constant) ranks.

Here we focus on the \emph{hierarchical semi-separable (HSS)} format among other types of hierarchical matrices. Technically speaking, the HSS format means $S_{\lc\rc}$ and  $S_{\rc\lc}$ are both zero and the bases $U_{\lc\rc}$/$V_{\lc\rc}$ and $U_{\rc\lc}$/$V_{\rc\lc}$ of a node in $\cal{T}$ are recursively defined through the bases of the node's children, i.e., the so-called \emph{nested bases}. 

We refer interested readers to~\cite{chenhan-biros-e17,yu-reiz-biros18,gofmm-home-page} for details about the \gofmm{} method.

\subsection{Summary \& contrast with related work} \label{s:compare}

We summarize the storage and computational complexity of our \hmatrix{} approximation method (HM), and describe its relation with three existing methods, namely, the Hessian-free method (HF)~\cite{martens2010deep,martens2011learning}, the randomized singular value decomposition (RSVD)~\cite{halko-martinsson-tropp11} and the Kronecker-factored Approximate Curvature (K-FAC)~\cite{martens16,grosse-martens15}. As before, we assume the MLP network has $L$ layers of constant layer sizes $d$, so the number of weights is $N=d^2L$. Let $n$ be the number of data points.




\paragraph{HM} The algorithm is given in \cref{alg:gofmm}, where the first three step requires $\bigO(N n d)$ work and $\bigO(N n)$ storage. Suppose the rank is $r_o$ in the \hmatrix{} approximation. \red{The \gofmm{} needs to call \cref{alg:sampling} $\bigO(N r_o)$ times, which results in $\bigO((n+d/\epsilon^2) N r_o)$ work. Here, $\epsilon$ is chosen to be around the same accuracy as the \hmatrix{} approximation with rank $r_o$.} In addition, standard results in the HSS literature~\cite{martinsson2005fast,xia2010fast,ho2013hierarchical} states that the factorization requires $\bigO(N r_o^2)$ work and $\bigO(N r_o)$ storage, which can be applied to solving a linear system with $\bigO(N r_o)$ work. 


\paragraph{MF} Unlike the other three methods, the MF does not approximate the GNH. It takes advantage of the (exact) matrix-free matvec and utilizes the conjugate gradient (CG) method for solving linear systems. It is based on the two primitives in \cref{tab:cg}, where every iteration costs $\bigO(N n)$ work and storage. The number of CG iteration is generally upper bounded by $\bigO(\sqrt{\kappa})$, where $\kappa$ is the condition number of the (regularized) GNH matrix. 

\paragraph{RSVD} Recall the GNH matrix $H = J^T Q J$. Without loss of generality, assume $Q$ is an identity for ease of description. The algorithm is to compute an approximate SVD of $J$ with the following steps, which natually leads to an approximate eigenvalue decomposition of $H$. First, we apply the back-propagation in \cref{tab:cg} with a random Gaussian matrix as input. Second, the QR decomposition of the result is used to estimate the row space of $J$. Third, the linearized forward is applied to project $J$ onto the approximate row space, and finally, the SVD is computed on the projection. Overall, the storage is $\bigO(N r)$, and the work required is $\bigO(N n r + N r^2 + dnr^2)$, where $r$ is the numerical rank from the QR decomposition. Compared with the HM approximating off-diagonal blocks, the RSVD approximates the entire matrix.

\paragraph{K-FAC} 
It computes an approximation of the Fisher matrix $F$ (mathematically equivalent to the GNH for some popular loss functions). Let a column vector $g=[\texttt{vec}(g^{1}),\ldots,\texttt{vec}(g^{L})] \in \mathbb{R}^{N}$ be the gradient, and $F = \E [g \, g^T]$ be a $L$-by-$L$ block matrix with block size $d^2$-by-$d^2$. Note the expectation here is taken with respect to both the empirical input data distribution $\hat{Q}_{x^0}$ and the network's predictive distribution $P_{y|\xl}$. In particular, the $(\ell_1, \ell_2)$-th block ($\ell_1, \ell_2 = 1,2,\ldots,L$) is given by
\begin{align}
F_\text{block}(\ell_1, \ell_2) \nonumber
=& \E[ \mbox{vec}(g^{\ell_1}) \mbox{vec}(g^{\ell_2})^T ] \nonumber \\
=& \E[ M^{\ell_1} z^{\ell_1} (x^{\ell_1-1})^T \left( M^{\ell_2} z^{\ell_2} (x^{\ell_2-1})^T \right)^T ] \label{kfac1} \\
=& \E[ (M^{\ell_1} z^{\ell_1} \otimes x^{\ell_1-1}) \left( M^{\ell_2} z^{\ell_2} \otimes x^{\ell_2-1} \right)^T ] \label{kfac2}  \\
=& \E[ (M^{\ell_1} z^{\ell_1} \otimes x^{\ell_1-1}) \left( (M^{\ell_2} z^{\ell_2})^T \otimes (x^{\ell_2-1})^T \right) ] \label{kfac3}  \\
=& \E[ M^{\ell_1} z^{\ell_1} (M^{\ell_2} z^{\ell_2})^T \otimes x^{\ell_1-1} (x^{\ell_2-1})^T ] \label{kfac4}  \\
\approx & \E[ M^{\ell_1} z^{\ell_1} (M^{\ell_2} z^{\ell_2})^T] \otimes \E[ x^{\ell_1-1} (x^{\ell_2-1})^T ] \label{kfac5} 
\end{align}
where Eq.~\eqref{kfac1} uses the definition of the network gradient in \cref{tab:cg}, Eq.~\eqref{kfac2} rewrites the equation using Kronecker products, Eq.~\eqref{kfac3} and Eq.~\eqref{kfac4} use the properties of Kronecker product, and Eq.~\eqref{kfac5} assumes the statistical independence between the two terms (see Section 6.3.1 in \cite{martens16}). In Eq.~\eqref{kfac5}, the former expectation is taken with respect to both $\hat{Q}_{x^0}$ and $P_{y|\xl}$, and the latter is taken with respect to $\hat{Q}_{x^0}$. To compute the first expectation, $k$ samples are drawn from the distribution $P_{y|x} \propto \text{exp}(-f(x_i^L, y))$, where $\xL_i$ is the network's output corresponding to input $x^0_i$. 
In practice, an additional block-diagonal or block-tridiagonal approximation of the inverse is employed for fast solution of linear systems. The main cost of the algorithm is constructing, updating and inverting $\bigO(L)$ matrices of size $d$-by-$d$, which requires $\bigO(d^2 L) = \bigO(N)$ storage and $\bigO( (nk+d) N)$ work. Overall, the approximation error of K-FAC has three components: the error of making the assumption \eqref{kfac5}, the sampling error from approximating the expectations in \eqref{kfac5} and the error of block-diagonal or block-tridiagonal approximation of the inverse.

We summarize the asymptotic complexities of the four methods discussed above in \tabref{t:complexity}. 

\begin{table}[htbp]
\centering
\caption{\red{Asymptotic complexities of the MF, the RSVD, the K-FAC and the HM with respect to the number of weights $N$ and the data size $n$ (``lower order'' terms not involving $Nn$ are dropped). We assume $r,k,r_o,d < n$, where $r$, $k$ and $r_o$ are the parameters in the RSVD, the K-FAC and the HM, respectively, and $d$ is the (constant/average) layer size. In addition, $\kappa$ stands for the condition number of the GNH matrix.}} \label{t:complexity}
\begin{threeparttable}
\begin{tabular}{|c|cccc|}
\hline
     & {\bf MF} & {\bf RSVD} & {\bf K-FAC} & {\bf HM} \\
\hline
\emph{construction}  & -          & $\bigO(N n r)$    & $\bigO(Nnk)$ & \red{$\bigO(N n (r_o+d))$} \\
\emph{storage}  & $\bigO(N n)$ & $\bigO(N r)$    & $\bigO(N)$  & $\bigO(N n)$\\
\emph{solve} &  $\bigO(Nn\sqrt{\kappa})$      & $\bigO(Nr)$    & $\bigO(N d)$ & $\bigO(N r_o)$\\
\hline
\end{tabular}
\end{threeparttable}
\end{table}

\section{Experimental Results} \label{s:results} 

In this section, we show (1) the cost and the accuracy of our \hmatrix{} approximations, (2) the memory savings from using the precomputation algorithm ($\bigO(N^2) \rightarrow \bigO(Nn)$), and (3) the efficiency of the fast sampling algorithm. \red{In Algorithm 4.1, the first two steps (precomputation) are implemented in Matlab for the convenience of extracting intermediate values of neural networks, and the last two steps (sampling) are implemented in C++ (\gofmm{} is written in C++).}

\paragraph{Networks and datasets}
We focus on classification networks and autoencoder networks with the MNIST and CIFAR-10 datasets. In the following, we denote networks' layer sizes as $d_1$\rar$d_2$\rar$\ldots$\rar$d_L$ from the input layer to the output layer. Every network has been trained using the stochastic gradient descent for a few steps, so the weights are not random.
\begin{enumerate}\zapspace
\item {``classifier''}: classification networks with the ReLU activation and the cross-entropy loss.
\begin{enumerate}
\zapspace
\setlength{\itemindent}{-.3in}
\item $N$={15,910}; MNIST dataset; layer sizes: 784\rar20\rar10. 
\item $N$={61,670}; CIFAR-10 dataset; layer sizes: 3072\rar20\rar10.
\item \red{$N$={219,818}; MNIST dataset; layer sizes: 784\rar256\rar64\rar32\rar10.} 
\item \red{$N$={1,643,498}; CIFAR-10 dataset; layer sizes: 3072\rar512\rar128\rar32\rar10.}
\end{enumerate}
%
%
%
%
\item {``AE''}: autoencoder networks with the softplus activation (sigmoid activation at the last layer) and the mean-squared loss.
\begin{enumerate}
\zapspace
\setlength{\itemindent}{-.3in}
\item $N$={16,474}; MNIST dataset; layer sizes: 784\rar10\rar784. 
\item $N$={64,522}; CIFAR-10 dataset; layer sizes: 3072\rar10\rar3072.
\item $N$={125,972}; CIFAR-10 dataset; layer sizes: 3072\rar20\rar3072.
\end{enumerate}
%
\end{enumerate}

\paragraph{\gofmm{} parameters} We \red{employ the default ``angle'' distance metric in \cref{e:angle}} and focus on three parameters in the \gofmm{} that control the accuracy of the \hmatrix{} approximation: (1) the leaf node size $m$ of the hierarchical partitioning $\cal{T}$ (equivalent to setting the number of tree levels), (2) the maximum rank $r_o$ of off-diagonal blocks, and (3) the accuracy $\tau$ of low-rank approximations. In particular, we ran \gofmm{} with two different accuracies: ``\texttt{low}'' ($m=128$, $r_o=128$, $\tau=\accnum{5E-2}$) and ``\texttt{high}'' ($m=1024$, $r_o=1024$, $\tau=\accnum{1E-5}$).

\paragraph{\gofmm{} results} We report the following results for our approach.
\begin{itemize}
\item
$t_\mathrm{build}$: time of constructing the \hmatrix{} approximation of the GNH matrix (\red{not including precomputation time}).
\item
$t_\mathrm{matv}$: time of applying the \hmatrix{} approximation to 128 random vectors.
\item
\%K: compression rate of the \hmatrix{} approximation, i.e., ratio between the \hmatrix{} storage and the GNH matrix storage.
 \item
 $\epsilon_F$: relative error of the \hmatrix{} approximation measured in Frobenius norm, estimated by $\| H x - {H}_{\gofmm{}} \, x \|_F / \| H x \|_F$, where $x \in \mathbb{R}^{N \times 128}$ is a Gaussian random matrix.
\end{itemize}

\subsection{Cost and accuracy of \hmatrix{} approximation} \label{s:acc_hmat}
\tabref{tab:GOFMMcomp} shows results of our \hmatrix{} approximations for networks that have relatively small numbers of parameters. The GNH matrices are computed and fully stored in memory.

As \tabref{tab:GOFMMcomp} shows, the approximation can achieve four digits' accuracy except for one network (two digits) when the accuracy of low-rank approximations is 1E-5. Since we have enforced the maximum rank $r_o$, the runtime of constructing \hmatrix{} approximations ($t_{\mathrm{build}}$) increases proportionally to the number of network parameters, and the compression rate scales inverse proportionally to the number of parameters. \red{The reported construction time $t_{\mathrm{build}}$ includes the cost of creating an implicit tree data structure $\cal{T}$ in \gofmm{}, which is less than 20\% of $t_{\mathrm{build}}$.} In addition, applying the \hmatrix{} approximations to 128 random vectors took less than one second for the five networks. These \hmatrix{} approximations can be factorized in linear time for solving linear systems and eigenvalue problems.



\begin{table}[htbp]
    \centering \small 
    \caption{\em Timing ($t_{\mathrm{build}}$ and $t_{\mathrm{matv}}$), compression rate (\%K) and accuracy ($\epsilon_F$) of \hmatrix{} approximations corresponding to low- and high-accuracy settings, respectively. $t_{\mathrm{build}}$ is the time of applying the \gofmm{} on the GNH matrices corresponding to 1000 data points. Experiments performed on one node from the Texas Advanced Computing Center ``Stampede 2'' system, which has two sockets with 48 cores of Intel Xeon Platinum 8160/``Skylake'' and 192 GB of RAM.}
    \label{tab:GOFMMcomp}
    \begin{tabular}{rcccrrrc} 
    \toprule
    \# & network & $N$ & \texttt{accuracy} &  $t_{\mathrm{build}}$ & $t_{\mathrm{matv}}$ & {\%K} & $\epsilon_F$ \\ 
    \midrule 
    \rownumber & classifier (a) & 16k & \texttt{low} & \num{0.24} & \num{0.03} & 1.80\% & \accnum{1.5E-1} \\
    \rownumber &  & & \texttt{high} & \num{5.74} & \num{0.11} & 13.59\% &\accnum{4.4E-4} \\
    \midrule
    \rownumber & classifier (b)  & 61k & \texttt{low} & \num{0.42} & \num{0.08} & 0.57\% &  \accnum{4.7E-1} \\
    \rownumber &  &  &   \texttt{high} & \num{13.28} & \num{0.33} & 4.78\%& \accnum{4.0E-2} \\
    \midrule
    \rownumber & AE (a) & 16k  & \texttt{low} & \num{0.27} & \num{0.03} &1.25\% & \accnum{1.2E-01} \\
    \rownumber &  &  &   \texttt{high} & \num{5.67} & \num{0.10} &11.38\% & \accnum{6.5E-04} \\
    \midrule
    \rownumber & AE (b) & 64k  & \texttt{low} & \num{0.43} & \num{0.08} &0.53\% & \accnum{5.5E-03} \\
    \rownumber &  & &  \texttt{high} & \num{13.26} & \num{0.38} &4.62\% & \accnum{6.6E-04} \\
    \midrule
    \rownumber & AE (c) & 126k   & \texttt{low} & \num{0.87} & \num{0.17} &0.28\% & \accnum{4.1E-03} \\
    \rownumber &  &  & \texttt{high} & \num{24.10} & \num{0.94} &2.32\% & \accnum{5.2E-04} \\
    \bottomrule
    \end{tabular}
\end{table}


\paragraph{Comparison with RSVD and K-FAC}


We \red{implemented the RSVD using Keras and TensorFlow for fast backpropogation, and we implemented the K-FAC in Matlab for the convenience of extracting intermediate values.}
\tabref{tab:compare} shows the accuracies of our method (HM), the RSVD and the K-FAC under about the same compression rate for the low- and high-accuracy settings, respectively. For the RSVD, the storage is $rN$ entries, where  $r$ is the numerical rank of the (symmetric) GNH matrix, so the compression rate is $r/N$. For the K-FAC, we use the relatively more accurate block tridiagonal version. The compression rate of the K-FAC is defined as $k/N$, and the reason is that the construction of the RSVD and the K-FAC requires the same number of back-propagation if $k=r$ (recall~\tabref{t:complexity}). So we choose $k$ and $r$ to be the same value such that the corresponding compression rate of the RSVD and the K-FAC are slightly higher than the HM. 

\begin{table}[htbp]
\centering
\caption{\em Comparison of accuracies ($\epsilon_F$) among the \hmatrix{} approximation (HM), the RSVD and the K-FAC with about the same compression rate (\%K) for the low- and high-accuracy settings, respectively. For the RSVD and the K-FAC, the compression rate means $r/N$ and $k/N$, respectively, where $r$ is the rank and $k$ is the number of random samples. For all cases, the GNH matrices correspond to 10,000 data points ($n=10,000$ in Eq.~\eqref{e:derivatives}).} 
\label{tab:compare}
\begin{tabular}{  c c  l  l  p{12mm}   p{12mm}   p{14mm}   p{14mm} } \toprule
 & & HM-low & HM-high & RSVD-low & RSVD-high &  K-FAC-low & K-FAC-high \\ \midrule
AE(a) & \%K & 1.23\%   &  11.77\%   &  1.40\%   &   12.14\%  &   1.40\%    &   12.14\%    \\
& $\epsilon_F$ & 1.7E-1  & 4.7E-4  &  4.3E-1  &  5.1E-3  & 1.2E-1  &  7.3E-2   \\ \midrule
AE(b) & \%K & 0.53\%   &  4.62\%   &  0.62\%   &  4.65\%   &    0.62\%   &   4.65\%    \\
& $\epsilon_F$ & 5.7E-3  & 6.4E-4  &  8.4E-1  &  2.3E-1  & 1.7E-1  &  3.8E-2    \\ \midrule
AE(c) & \%K & 0.28\%   &  2.31\%   &  0.32\%   &  2.38\%   &  0.32\%     &    2.38\%   \\
& $\epsilon_F$ & 4.2E-3  & 4.9E-4  &  9.1E-1  &  2.1E-1  &  1.6E-1   &   4.1E-2     \\
\bottomrule
\end{tabular}
\end{table}

As \tabref{tab:compare} shows, the \hmatrix{} approximation achieved higher accuracy than the RSVD and the K-FAC for most cases, especially for the high accuracy setting. For the RSVD, suppose the eigenvalues of the GHN matrix are $\{\sigma_i \}_i$, and the error of the rank-$k$ approximation measured in the Frobenius norm is proportional to $(\sum_{i>k} \sigma_i^2)^{1/2}$. For autoencoder (b) and (c), the spectrums of the GNH matrices decay slowly, so the RSVD is not efficient. For the K-FAC, the approximation that the expectation of a Kronecker product equals to the Kronecker product of expectations (Eq.~\eqref{kfac5}) is, in general, not exact, impeding the overall accuracy of the method.

\subsection{Memory savings} \label{s:memory}

\tabref{tab:GOFMMCompresJac} shows the memory footprint between our precomputation Eq.~\eqref{e:c} and the full GNH matrix, i.e., $\bigO(N^2)$. Recall \cref{th:th1} that the storage of our precomputation is $\bigO(nN)$, where $n$ is the number of data points.

As \tabref{tab:GOFMMCompresJac} shows, our precomputation leads to huge memory reduction compared with storing the full GNH matrix. This allows using the \gofmm{} method for networks that have a large number of parameters. For example, the storage of the GNH matrix for classifier (d) network requires more than 10 TB! But we were able to run \gofmm{} with the compressed storage (at the price of spending $\bigO(dn)$ work for the evaluation of every entry). \red{The precomputation of Eq.~\eqref{e:c} took merely about 2s and 7s, respectively.}

\begin{table}[htbp]
    \centering \small 
    \caption{\em Comparison of memory footprint (in single-precision) between our precomputation (Eq.~\eqref{e:c}) and the full GNH matrix. For each network, we show the compression rate and accuracy of \hmatrix{} approximations for two levels of accuracies. For both cases, the GNH matrices correspond to \red{$n=10,000$ data points}.}
    \label{tab:GOFMMCompresJac}
    \begin{tabular}{cccccrc} 
    \toprule
      & $N$ & $M_{\mathrm{ours}}$ & $M_{\mathrm{GNH}}$ & \texttt{accuracy} &    {\%K} & $\epsilon_F$  \\ 
    \midrule 
     classifier (c) & 219k &  191 MB & 193 GB &  \texttt{low}  & 0.165\% & \accnum{3.1E-01} \\
                         &          &&                                   & \texttt{high}  & 1.268\% & \accnum{4.2E-02} \\ \midrule
     classifier (d) & 1.6m  & 423 MB & 10.8 TB &  \texttt{low}  & 0.012\% & \accnum{1.2E-01} \\
                         &          &&                                   & \texttt{high}  & 0.177\% & \accnum{2.2E-02} \\
    \bottomrule
    \end{tabular}
\end{table}

\subsection{Fast Monte Carlo sampling} \label{s:res_sample}
We show the accuracy of our fast Monte Carlo sampling scheme.
The relative error measured in the Frobenius norm is between the exact GNH matrix $H$ and the approximation $\tilde{H}$ computed using \cref{alg:sampling} with a prescribed number of random samples. For reference, we also run the same sampling scheme but with a uniform probability distribution.

\begin{table}[htbp]
\centering
\caption{\em Accuracy of our fast Monte Carlo (FMC) sampling scheme. The error $\epsilon_K = \| H - \tilde{H} \|_F / \| H \|_F$, where $H$ and $\tilde{H}$ are the exact GNH matrix and its approximation computed using $K$ random samples in \cref{alg:sampling}, respectively. The exact GNH matrices correspond to \red{the AE (a) network with the entire MNIST dataset and the class (a) network with the entire CIFAR-10 training dataset, respectively.}
 The reference uniform sampling scheme uses a uniform sampling probability instead of Eq.~\eqref{e:probability} in \cref{alg:sampling}.}
\label{tab:element}
\begin{tabular}{ccccccc}
\toprule
& $n$  & \texttt{scheme}& $\epsilon_{10}$ & $\epsilon_{100}$  & $\epsilon_{1,000}$ &  $\epsilon_{10,000}$  \\ \midrule
MNIST & 60,000 & uniform & \accnum{0.3581} & \accnum{0.1127} & \accnum{0.0360} & \accnum{0.0113}  \\
            && FMC     &  \accnum{9.7e-3} & \accnum{3.1e-3} & \accnum{9.6e-4} & \accnum{3.1e-4} \\ \midrule
CIFAR-10 & 50,000 & uniform & \accnum{0.9731} & \accnum{0.3076} & \accnum{0.09812} & \accnum{0.03611}  \\
            && FMC     &  \accnum{0.6104} & \accnum{0.189} & \accnum{0.0611} & \accnum{0.0193} \\            
\bottomrule
\end{tabular}
\end{table}

As \tabref{tab:element} shows, when the number of random samples increases by $100\times$, the accuracy improves by $10\times$, which confirms the standard convergence rate of Monte Carlo in \cref{th:th2}. Importantly, the error bound and the convergence rate do \emph{not} depend on the problem size $n$. Moreover, our sampling scheme outperforms the uniform sampling by at most two orders of magnitude for the MNIST dataset. In other words, the uniform sampling requires $100\times$ more random samples to achieve about the same accuracy as our sampling scheme.






\paragraph{\hmatrix{} approximation with sampling} 
\tabref{tab:sampling} shows the error of \cref{alg:gofmm} for a sequence of increasingly large number of random samples. Recall that \cref{alg:gofmm} computes the \hmatrix{} approximation $\tilde{H}_{\gofmm{}}$ for the (inexact) GNH matrix, namely $\tilde{H}$ from \cref{alg:sampling} . The error between the \hmatrix{} approximation $\tilde{H}_{\gofmm{}}$ and the exact GNH matrix, namely $H$ is bounded as below \red{(using the triangular equality)}
\[
\| H - \tilde{H}_{\gofmm{}} \|_F = \| H - \tilde{H} + \tilde{H} - \tilde{H}_{\gofmm{}}  \|_F
\le \| H - \tilde{H} \|_F +  \| \tilde{H}  - \tilde{H}_{\gofmm{}} \|_F,
\]
where the first term is the sampling error, and the second term is the \hmatrix{} approximation error. As \tabref{tab:element} shows, the former converges to zero and is independent of the data size. \tabref{tab:sampling} shows that the latter also converges as the sampling becomes increasingly accurate, which justifies the overall approach.

\begin{table}[h]
     \caption{\em \hmatrix{} approximation with sampling. The exact GNH matrices correspond to \red{the AE (a) network with the entire MNIST dataset and the class (a) network with the entire CIFAR-10 training dataset, respectively}. The compression rate and the accuracy are shown for low- and high-accuracy settings, respectively.}
     \label{tab:sampling}
    \begin{subtable}[h]{1.0\textwidth}
        \centering
        \caption{\red{MNIST dataset ($n=60,000$ images)}}
        \label{tab:week1}
	\begin{tabular}{crcrcrcrc}
        \toprule 
        & \multicolumn{2}{c}{10 samples}  & \multicolumn{2}{c}{100 samples} & \multicolumn{2}{c}{1,000 samples} & \multicolumn{2}{c}{10,000 samples}  \\ \midrule
        \texttt{accuracy} &  {\%K} & $\epsilon_F$ &  {\%K} & $\epsilon_F$ &  {\%K} & $\epsilon_F$ &  {\%K} & $\epsilon_F$ \\ \midrule 
         \texttt{low} & 1.74\% & \accnum{2.8E-1} & 1.69\% & \accnum{1.3E-1} & 1.68\% & \accnum{1.4E-1}  & 1.68\% &\accnum{6.1E-2}  \\
          \texttt{high} & 17.1\% &\accnum{2.2E-2} & 16.9\% &\accnum{9.5E-3} & 16.6\% & \accnum{2.7E-3} & 16.2\% &\accnum{9.6E-4}  \\
        \bottomrule
        \end{tabular}       
    \end{subtable}
    \hfill
    \begin{subtable}[h]{1.0\textwidth}
        \centering
        \caption{\red{CIFAR-10 training dataset ($n=50,000$ images)}}
        \label{tab:week2}
        	\begin{tabular}{crcrcrcrc}
        \toprule 
        & \multicolumn{2}{c}{10 samples}  & \multicolumn{2}{c}{100 samples} & \multicolumn{2}{c}{1,000 samples} & \multicolumn{2}{c}{10,000 samples}  \\ \midrule
        \texttt{accuracy} &  {\%K} & $\epsilon_F$ &  {\%K} & $\epsilon_F$ &  {\%K} & $\epsilon_F$ &  {\%K} & $\epsilon_F$ \\ \midrule 
         \texttt{low} & 0.61\% & \accnum{9.7E-1} & 0.61\% & \accnum{5.3E-1} & 0.61\% & \accnum{4.1E-1}  & 0.61\% &\accnum{4.3E-1}  \\
          \texttt{high} & 4.83\% &\accnum{9.7E-1} & 4.83\% &\accnum{3.6E-1} & 4.83\% & \accnum{1.9E-1} & 4.83\% &\accnum{7.3E-2}  \\
        \bottomrule
        \end{tabular}        
     \end{subtable}
\end{table}

\section{Conclusions} \label{s:conclusions} \label{s:conclusion} 
We have presented a fast method to evaluate entries in the GNH matrix of the MLP network, and our method is consisted of two parts: a precomputation algorithm and a fast Monte Carlo algorithm. While the precomputation allows evaluating entries in the GNH matrix exactly with reduced storage, the random sampling is based on the precomputation and further accelerates the evaluation. Let $N$ be the number of weights, $n$ be the data size, and $d$ be the constant layer size. \red{Our scheme requires $\bigO(n+d/\epsilon^2)$ work for any entry in the GNH matrix $H$, where $\epsilon$ is the accuracy, whereas the worst case complexity to evaluate an entry exactly is $\bigO(Nn)$ through the matrx-free matvec. For example, the evaluation of $H_{N,N}$ would require $\bigO(Nn)$ work, while given our precomputation, it requires only $\bigO(n)$ work to compute a diagonal entry exactly (\cref{rmk}). One application of this fast diagonal evaluation would be computing all the diagonals of $H$ to precondition/accelerate the training of neural networks~\cite{yao2020adahessian}.} In this paper, we focused on applying the \gofmm{} to construct the \hmatrix{} approximation for the GNH matrix. As a result, we obtain an \hmatrix{} and its factorization for solving linear systems and eigenvalue problems with the GNH.


Two important directions for future research are (1) extending our method to other types of networks such as convolutional networks, where weight matrices are highly structured, (preliminary experiments on the VGG network show similar results as those in~\tabref{tab:GOFMMcomp}) 
and (2) incorporating our method in the context of a learning task, which would also require several algorithmic choices related to optimization, such as initialization, damping and adding momentum. 

\appendix 

\section{\red{McDiarmid’s Inequality}}\label{s:app}

\begin{theorem} 
Let $X_1,X_2,\ldots,X_n$ be independent random variables taking values in the set $\mathcal{X}$. If a mapping $F: \mathcal{X}^n \rightarrow \mathbb{R}$ satisfies
\[
\sup_{x_1,\ldots,x_i,x_i',\ldots,x_n} \left| F(x_1,\ldots,x_i,\ldots,x_n) - F(x_1,\ldots,x_i',\ldots,x_n) \right| \le \Delta_i, \quad \forall i
\]
where $x_1,\ldots,x_i,x_i',\ldots,x_n \in \mathcal{X}$, then for all $\epsilon > 0$,
\[
 \mbox{Pr}( f-\E[f] \ge \epsilon) \le  \mbox{exp} \left( {-2 \epsilon^2 \over \sum_{i=1}^n \Delta_i^2} \right).
\] 
\end{theorem}


\bibliographystyle{siamplain}
\bibliography{M129961}

\end{document}